\documentclass[a4paper,oneside,11pt]{amsproc} 

\usepackage[ansinew]{inputenc} 
\usepackage[T1]{fontenc}
\usepackage[english]{babel} 

\usepackage{enumerate} 
\usepackage{booktabs}

\usepackage{amsmath} 
\usepackage{amsfonts} 
\usepackage{amssymb}
\usepackage{amsthm}
\usepackage{latexsym}
\usepackage{setspace}
\usepackage[top=3cm , bottom=3cm , left=3cm , right=3cm]{geometry}
\usepackage[pdftex]{graphicx}
  \DeclareGraphicsExtensions{.pdf,.png,.jpg,.jpeg,.mps}
  \usepackage{pgf}
  \usepackage{tikz}
\usepackage{floatflt}
\usepackage{mathrsfs}

\makeatletter

\@namedef{subjclassname@2010}{%

  \textup{2010} Mathematics Subject Classification}

\makeatother

\theoremstyle{plain}
\newtheorem{prop}{Proposition}
\newtheorem{lemma}[prop]{Lemma}
\newtheorem*{theorem}{Theorem}
\newtheorem*{cor}{Corollary}
\newtheorem*{hc}{Hypercontractivity Theorem}

\theoremstyle{definition} 
\newtheorem*{def*}{Definition}

\theoremstyle{remark}
\newtheorem*{remark}{Remark}
\newtheorem*{acks}{Acknowledgements}

\newcommand{\R}{\mathbb{R}}

\title[Admissible decomposition on Gaussian $L^p$]{Admissible decomposition for spectral multipliers on Gaussian $L^p$}
\author{Mikko Kemppainen}
\begin{document}

\begin{abstract}
  This paper concerns harmonic analysis of
  the Ornstein--Uhlenbeck operator $L$ on the Euclidean space.
  We examine the method of decomposing a spectral multiplier
  $\phi(L)$ into three parts according to the notion of
  admissibility, which quantifies the doubling behaviour of the
  underlying Gaussian measure $\gamma$.
  We prove that the above-mentioned admissible decomposition
  is bounded in $L^p(\gamma)$ for $1 < p \leq 2$ in a certain sense
  involving the Gaussian conical square function.
  The proof relates admissibility with 
  E. Nelson's hypercontractivity theorem in a novel way.
\end{abstract}


\maketitle

%\tableofcontents

\section{Introduction}

\subsection{General background}
This article is a continuation of \cite{Kem16}, regarding analysis of the \emph{Ornstein--Uhlenbeck operator}
\begin{equation*}
  L = -\frac{1}{2} \Delta + x\cdot\nabla ,
\end{equation*}
which on the Euclidean space $\mathbb{R}^n$ is associated with the \emph{Gaussian measure}
\begin{equation*}
  d\gamma (x) = \pi^{-n/2} e^{-|x|^2} \, dx .
\end{equation*}

In \cite{Kem16}, a certain class of spectral multipliers 
$\phi(L)$ was studied by means of an \emph{admissible decomposition}
--- an integral representation, which takes into account the
non-doubling behaviour of $\gamma$. This representation allows us to
express the multiplier as a sum of three parts (admissible, 
intermediate, and non-admissible):
\begin{equation*}
  \phi(L)f = c(\pi_1u + \pi_2f + \pi_3f) ,
\end{equation*}
where $c$ is a constant and $u$ arises from $f$. An $L^1$-estimate 
was then obtained in terms of an 
\emph{admissible conical square function} $Sf$, namely,
\begin{equation*}
  \| \pi_1u \|_1 \lesssim \| Sf \|_1 , \quad 
  \| \pi_2f \|_1 \lesssim \| f \|_1 , \quad %\textup{and} \quad
  \| \pi_3f \|_1 \lesssim \| (1 + \log_+ |\cdot |) \, Mf \|_1 ,
\end{equation*}
but the third estimate with a logarithmic weight and a maximal
function $Mf$ is clearly unsatisfactory. This shortcoming 
calls into question
whether the admissible decomposition is at all suitable %a reasonable approach
for studying boundedness of spectral multipliers.
On the other hand, such problems do not seem to
appear in \cite{Por14}, from which the decomposition
originates in connection with the Riesz transform $\nabla L^{-1/2}$.

The role of this article is to justify the above-mentioned
approach, and to serve
as an intermediate step towards a fully satisfactory $L^1$-estimate. 
Indeed, we show here that for $1 < p \leq 2$ we have
\begin{equation*}
  \| \pi_1u \|_p \lesssim \| Sf \|_p , \quad 
  \| \pi_2f \|_p \lesssim \| f \|_p , \quad %\textup{and} \quad
  \| \pi_3f \|_p \lesssim \| f \|_p .
\end{equation*}
Interestingly, the proof of the third estimate 
invokes the hypercontractivity theorem of E. Nelson 
\cite{Nel73},
and relies on its subtle interplay with the concept of admissibility. 
The ultimate aim of this square function 
approach is to provide a metric theory of Gaussian Hardy spaces
to complement the existing atomic theory \cite{MM07}.

\subsection{Admissible conical square function}

Recall that the \emph{admissibility function}
\begin{equation*}
  m(x) = \min (1 , |x|^{-1}) , \quad x\in\mathbb{R}^n ,
\end{equation*}
quantifies the extent to which $\gamma$ is doubling:
\begin{equation*}
  \gamma(B(x,2t)) \lesssim \gamma(B(x,t)) , \quad t\leq m(x) .
\end{equation*}
See \cite{MM07,MvNP12,AK15} for more details.

The admissible conical square function is then defined by
\begin{equation*}
  Sf(x) = \Big( \int_0^{2m(x)} \frac{1}{\gamma (B(x,t))} \int_{B(x,t)} |t^2L e^{-t^2L}f(y)|^2 
  \, d\gamma (y) \, \frac{dt}{t} \Big)^{1/2} , \quad x\in\mathbb{R}^n,
\end{equation*}
where the diffusion semigroup
\begin{equation*}
  e^{-tL}f(x) = \int_{\mathbb{R}^n} M_t(x,y) f(y) \, d\gamma (y) , \quad t>0 ,
\end{equation*}
is given by the \emph{Mehler kernel}
\begin{equation*}
  M_t(x,y) = \frac{1}{(1-e^{-2t})^{n/2}}
  \exp \Big( -\frac{e^{-t}}{1-e^{-2t}} |x-y|^2 \Big)
  \exp \Big( \frac{e^{-t}}{1+e^{-t}} (|x|^2 + |y|^2) \Big) .
\end{equation*}
The origins of this Gaussian square function can be found in
\cite{MvNP11,Por14}. The
Ornstein--Uhlenbeck semigroup $(e^{-tL})_{t>0}$ is a prototypical
example of a symmetric, contractive, and conservative diffusion
semigroup in the sense of \cite{Ste70}.
For more information, see the (old, but not obsolete) survey
\cite{Sjo97}.

\subsection{Class of spectral multipliers}
\label{phiclass}
 
We will consider spectral multipliers of the form
\begin{equation*}
  \phi(\lambda) = \int_0^\infty \Phi(t) (t\lambda)^2 e^{-t\lambda} \,
  \frac{dt}{t} , \quad \lambda\geq 0,
\end{equation*}
where $\Phi : (0,\infty) \to \mathbb{C}$ is twice continuously differentiable and satisfies
\begin{equation}
\label{phibounds}
  \sup_{0<t<\infty} (|\Phi(t)| + t|\Phi'(t)|) + \sup_{0<t\leq 1} |t^2 \Phi''(t)| < \infty .
\end{equation}
As explained in \cite{Kem16}, these are a special kind of
`Laplace transform type' multipliers.

Moreover, we will refer to the following two extra conditions.
\begin{itemize}
\item Condition D:
\begin{equation*}
  \int_1^\infty (|\Phi'(t)| + t|\Phi''(t)|) \, dt < \infty .
\end{equation*}
\item Condition P: There exists an integer $N$ such that
\begin{equation*}
  |\Phi'(t)| + t|\Phi''(t)| \lesssim t^N, \quad t\geq 1 .
\end{equation*}
\end{itemize}

Notice, however, that the main result is already interesting
for the prototypical \emph{imaginary powers} $\phi(L) = L^{i\tau}$,
$\tau\in\R$, with $\Phi(t) = t^{-i\tau} / \Gamma (2 - i\tau)$
(or for their damped versions with $\Phi(t) = t^{-i\tau} \chi(t)$,
where $\chi$ is a smooth cutoff with 
$1_{(0,1]} \leq \chi \leq 1_{(0,2]}$).

\subsection{Admissible decomposition}
The analysis is greatly simplified by switching to the
discretized version of the admissibility function
\begin{equation*}
  \widetilde{m}(x) =
  \begin{cases}
    1, &|x|<1, \\
    2^{-k}, &2^{k-1} \leq |x| < 2^k , \quad k\geq 1,
  \end{cases}
\end{equation*}
and to the associated \emph{admissible region}
$D = \{ (y,t)\in \mathbb{R}^n \times (0,\infty) : 0 < t < \widetilde{m}(y) \}$.

Let then $\phi$ and $\Phi$ be as in Subsection \ref{phiclass} 
and let $f$ be a polynomial with $\int f \, d\gamma = 0$.
The special form of our spectral multipliers allows us to use
the following integral representation with $\delta, \delta' > 0$
and $\kappa \geq 1$:
\begin{equation}
\label{admdec}
  \begin{split}
  \phi(L)f &= c_{\delta,\delta'} \int_0^\infty \Phi((\delta'+\delta)t^2) (t^2L)^2 
               e^{-(\delta'+\delta)t^2L} f \, \frac{dt}{t} \\
    &= c_{\delta,\delta'} \Big( \int_0^{\widetilde{m}(\cdot)/\kappa} \widetilde{\Phi}(t^2) t^2L
               e^{-\delta' t^2L} u(\cdot , t) \, \frac{dt}{t} \\
    &\quad\quad + \int_0^{\widetilde{m}(\cdot)/\kappa} \widetilde{\Phi}(t^2) t^2L
               e^{-\delta' t^2L} (1_{D^c}(\cdot , t) t^2L e^{-\delta t^2L}f) \, 
               \frac{dt}{t} \\
    &\quad\quad + \int_{\widetilde{m}(\cdot)/\kappa}^\infty \widetilde{\Phi}(t^2) (t^2L)^2 
               e^{-(\delta'+\delta)t^2L}f \, \frac{dt}{t} \Big) \\
    &=: c_{\delta,\delta'} ( \pi_1 u + \pi_2 f + \pi_3 f ),
    \end{split}
\end{equation}
where $u(\cdot , t) = 1_D(\cdot , t) t^2L e^{-\delta t^2L}f$
and $\widetilde{\Phi} (t) = \Phi ((\delta'+\delta)t)$.
The role of the technical parameters $\delta, \delta'$ and $\kappa$
is more visible in \cite{Kem16} than in this paper.

\subsection{Main result}

The first part of Proposition \ref{pi3} refines the previous
analysis of $\pi_3$ from \cite{Kem16}, and shows that the maximal
operator 
\begin{equation*}
  Mf(x) = \sup_{\varepsilon m(x)^2 < t \leq 1} |e^{-tL}f(x)|, \quad x\in\mathbb{R}^n ,
\end{equation*}
can be disposed of, i.e. that
\begin{equation*}
  \| \pi_3 f \|_1 \lesssim \| (1 + \log_+ |\cdot |) \, f \|_1 ,
\end{equation*}
for multipliers satisfying Condition D. As a consequence,
for all $f\in L^1(\gamma)$ it then holds that
\begin{equation*}
  \| \phi (L) f \|_1 \lesssim \| Sf \|_1 + \| (1 + \log_+ |\cdot |) \, f \|_1 .
\end{equation*}

The second part of Proposition \ref{pi3} 
(together with Propositions 2 and 4) leads to the main result of
the article:

\begin{theorem}
  Let $1 < p \leq 2$. For multipliers satisfying Condition P,
  there exist values of 
  parameters $\delta, \delta'$ and $\kappa$ so that
  \begin{equation*}
    \| \pi_1 u \|_p \lesssim \| Sf \|_p, \quad \| \pi_2 f \|_p \lesssim \| f \|_p, \quad \| \pi_3 f \|_p \lesssim \| f \|_p .
  \end{equation*}
\end{theorem}

\begin{cor}
  Let $1 < p \leq 2$. For spectral multipliers $\phi$
  of Subsection \ref{phiclass} satisfying Condition P we have
  \begin{equation*}
    \| \phi (L) f \|_p \lesssim \| Sf \|_p + \| f \|_p .
  \end{equation*}
\end{cor}

Such spectral multipliers are well known to be bounded on 
$L^p(\gamma)$ for all $1 < p < \infty$,
also in vastly more general settings
\cite{Ste70,Cow83,CD13}.
The vertical square function
that is typically used in their analysis seems, however,
to be somewhat ill-suited for $p=1$ and the corresponding
Hardy space theory. Developments of an abstract semigroup approach
to Hardy spaces nevertheless exist, see \cite{Mei08,JMP14}.
Recall also the relations between vertical and conical objects
in \cite[Proposition 2.1]{AHM12}, 
showing how conical square functions
dominate the vertical ones for $p\leq 2$. 
Moreover, it is curious to note that
a \emph{local} square function such as ours is sufficient
for the analysis of an operator with a spectral gap 
(between the lowest two eigenvalues in 
$\sigma (L) = \{ 0,1,2,\ldots \}$).
The intriguing question whether $\| Sf \|_p \lesssim \| f \|_p$
for $p > 1$ is a topic of ongoing research.

\section{Proof}

Throughout the proof we assume that $f$ is a polynomial with
$\int f \, d\gamma = 0$, and therefore a finite linear combination
of \emph{Hermite polynomials} --- the eigenfunctions of $L$.
The three parts of the admissible decomposition \eqref{admdec}
are studied separately in the following three subsections.

\subsection{Admissible part}

Let us first recall the definition of tent spaces (see \cite{AK15,MvNP12}).

\begin{def*}
Let $1\leq p \leq 2$. The \emph{Gaussian tent space} $\mathfrak{t}^p(\gamma)$ is defined to consist 
of functions $u$ on the admissible region 
$D = \{ (y,t)\in \mathbb{R}^n \times (0,\infty) : 0 < t < \widetilde{m}(y) \}$ 
for which
  \begin{equation*}
    \| u \|_{\mathfrak{t}^p(\gamma)} = \Big( \int_{\mathbb{R}^n} \Big( 
    \iint_{\Gamma (x)} |u(y,t)|^2 \, \frac{d\gamma (y) \, dt}{t\gamma (B(y,t))} 
    \Big)^{p/2} d\gamma (x) \Big)^{1/p} < \infty .
  \end{equation*}
  Here $\Gamma (x) = \{ (y,t)\in D : |y-x| < t \}$ is the admissible 
  cone at $x\in \mathbb{R}^n$.
\end{def*}

Consider the admissible part 
\begin{equation*}
  \pi_1 u = \int_0^{\widetilde{m}(\cdot)/\kappa} \widetilde{\Phi}(t^2) t^2L
               e^{-\delta' t^2L} u(\cdot , t) \, \frac{dt}{t}
\end{equation*}
for functions $u$ in a Gaussian tent space.

Curiously, due to the non-uniformity of the admissibility function,
the case $p=2$ is not quite as straightforward as one might expect.

\begin{prop}
  \label{pi1lemma}
  For $\kappa \geq 1$ and $0 < \delta' \leq 1$ we have
  $\| \pi_1 u \|_2 \lesssim \| u \|_{\mathfrak{t}^2(\gamma)}$.
\end{prop}
  \begin{proof}
    The proof does not rely on admissibility in the sense that 
    $\widetilde{m}(x)/\kappa$ can be replaced by any
    function with values in $(0,1]$. 
    Hence we may abbreviate $\widetilde{m}(x) / \kappa = m(x)$.    
     
    Write $\chi_t(x) = 1_{(0,m(x))}(t)$.
    Given a $g\in L^2(\gamma)$, we argue by duality:
    \begin{equation*}
    \begin{split}
      |\langle \pi_1 u , g \rangle | 
      &= \Big| \int_{\R^n} \int_0^{m(\cdot)} \widetilde{\Phi}(t^2) t^2L e^{-\delta' t^2L}u(\cdot , t) \,
      \frac{dt}{t} \, g \, d\gamma \Big| \\
      &= \Big| \int_0^1 \widetilde{\Phi}(t^2) \int_{\R^n} t^2L e^{-\delta' t^2L}u(\cdot , t)
      \chi_t g \, d\gamma \, \frac{dt}{t} \Big| \\
      &= \Big| \int_0^1 \int_{\R^n} u(\cdot ,t) t^2L e^{-\delta' t^2L} 
      (\chi_t g) \, d\gamma \, \frac{dt}{t} \Big| \\
      &\leq \| u \|_{\mathfrak{t}^2(\gamma)} 
      \Big( \int_0^1 \| t^2L e^{-\delta' t^2L} (\chi_t g) \|_2^2 \, \frac{dt}{t} \Big)^{1/2} ,
      \end{split}
    \end{equation*}
    and so it suffices to show that
    \begin{equation*}
      \Big( \int_0^1 \| t^2L e^{-\delta' t^2L} (\chi_t g) \|_2^2 \, \frac{dt}{t} \Big)^{1/2} \lesssim \| g \|_2 .
    \end{equation*}
    
    Now the uniform $L^2$-boundedness of $(t^2L)^{1/2} e^{-\frac{\delta'}{2}t^2L}$ guarantees that
    \begin{equation*}
    \begin{split}
      \| t^2L e^{-\delta' t^2L} (\chi_t g) \|_2
      &\lesssim \| (t^2L)^{1/2} e^{-\frac{\delta'}{2} t^2L} (\chi_t g) \|_2 \\
      &= \int_{\R^n} t^2L e^{-\delta' t^2L} (\chi_t g) \,
      \overline{\chi_t g} \, d\gamma ,
      \end{split}
    \end{equation*}
    where the last step relies on the self-adjointness and 
    non-negativity of $(t^2L)^{1/2}e^{-\delta' t^2L}$.
    Expressing $t^2Le^{-\delta' t^2L}$ in terms of the kernel 
    $(2\delta')^{-1}t\partial_t M_{\delta' t^2}(x,y)$ we therefore see
    that
    \begin{equation*}
    \begin{split}
      &\int_0^1 \int_{\R^n} |t^2L e^{-\delta' t^2L} (\chi_t g)|^2 \, d\gamma \, \frac{dt}{t} \\
      &\lesssim \Big| \int_0^1 \int_{\R^n} 
      \int_{\R^n} t \partial_t M_{\delta' t^2}(x,y)
      1_{(0,m(y))}(t)g(y) \, d\gamma (y) \, 1_{(0,m(x))}(t)
      \overline{g(x)} \, d\gamma (x)\, \frac{dt}{t} \Big| \\
      &= \Big| \int_{\R^n}\overline{g(x)} \int_{\R^n} g(y) 
      \int_0^{m(x) \wedge m(y)} \partial_t M_{\delta' t^2}(x,y)
      \, dt \, d\gamma (y) \, d\gamma (x) \Big| \\
      &= \Big| \int_{\R^n}\overline{g(x)} \int_{\R^n} 
      M_{\delta' (m(x)\wedge m(y))^2}(x,y)
      g(y) \, d\gamma (y) \, d\gamma (x) \Big| \\
      &\leq \Big| \int_{\R^n}\overline{g(x)} 
      \int_{\{ y: m(y)\leq m(x) \}} 
      M_{\delta' m(y)^2}(x,y)
      g(y) \, d\gamma (y) \, d\gamma (x) \Big| \\
      &\quad + \Big| \int_{\R^n} g(y) \int_{\{ x: m(x)\leq m(y) \}} 
      M_{\delta' m(x)^2}(y,x)
      \overline{g(x)} \, d\gamma (x) \, d\gamma (y) \Big| \\
      &\leq \int_{\R^n} |g(x)| \, \sup_{t>0} e^{-tL} |g|(x) \, 
      d\gamma (x)
      + \int_{\R^n} |g(y)| \, \sup_{t>0} e^{-tL} |g|(y) \, 
      d\gamma (y) \\
      &\leq 2 \int_{\R^n} (\sup_{t>0} e^{-tL}|g| )^2 \, d\gamma \\
      &\lesssim \| g \|_2^2 ,
      \end{split}
    \end{equation*}
    where in the last step we made use of the maximal inequality.
    This finishes the proof.
  \end{proof}
  
\begin{prop}
  \label{pi1}
  Let $1 < p \leq 2$.
  For $\kappa \geq 1$ and sufficiently small $\delta' > 0$, we have
  $\| \pi_1 u \|_p \lesssim \| u \|_{\mathfrak{t}^p(\gamma)}$. 
  Moreover, for $0 < \delta \leq 1$, the function 
  $u(\cdot , t) = 1_D(\cdot , t) t^2L e^{-\delta t^2L} f$ satisfies 
  $\| u \|_{\mathfrak{t}^p(\gamma)} \lesssim \| Sf \|_p$.
\end{prop}
\begin{proof}
  The first part of the statement follows by interpolation of
  Gaussian tent spaces 
  \cite[Theorem 3.3 and Corollary 3.5]{AK15}. Indeed,
  \begin{equation*}
    \pi_1 \text{ is bounded } 
    \begin{cases}
      \mathfrak{t}^2(\gamma) \to L^2(\gamma) \quad
      \text{(by Proposition \ref{pi1lemma} above)}, \\
      \mathfrak{t}^1(\gamma) \to L^1(\gamma) \quad
      \text{(by \cite[Proposition 2]{Kem16})} .
    \end{cases}
  \end{equation*}  
  Therefore, $\pi_1$ is also bounded 
  $\mathfrak{t}^p(\gamma) \to L^p(\gamma)$, i.e.
  $\| \pi_1 u \|_p \lesssim \| u \|_{\mathfrak{t}^p(\gamma)}$.
  
  The second part of the statement follows by a straightforward
  modification of the corresponding argument in
  \cite[Proposition 2]{Kem16}. Indeed, by change of aperture on
  $\mathfrak{t}^p(\gamma)$ (see \cite[Theorem 3.3]{AK15})
  we obtain
  \begin{equation*}
    \| u \|_{\mathfrak{t}^p(\gamma)}
    \lesssim \Big( \int_{\R^n}
    \Big( \iint_{\Gamma (x) \cap D'} |s^2Le^{-s^2L}f(y)|^2 \,
    \frac{d\gamma (y) \, ds}{s\gamma(B(y,s))} \Big)^{p/2} d\gamma (x)
    \Big)^{1/p} ,
  \end{equation*}
  where $D' = \{ (y,s)\in\R^n \times (0,\infty) : s < \sqrt{\delta} \widetilde{m}(y) \}$.
  The desired estimate
  $\| u \|_{\mathfrak{t}^p(\gamma)} \lesssim \| Sf \|_p$ 
  now follows from the pointwise inequality
  (see \cite[Proposition 2]{Kem16})
  \begin{equation*}
    \begin{split}
      &\iint_{\Gamma (x) \cap D'} 
      |s^2L e^{-s^2L}f(y)|^2 
      \,\frac{d\gamma (y)\, ds}{s \gamma (B(y,s))} \\
      &\lesssim \int_0^{2m(x)} \frac{1}{\gamma (B(x,s))}\int_{B(x,s)} |s^2L e^{-s^2L}f(y)|^2 
  \, d\gamma (y) \, \frac{ds}{s} , \quad x\in\R^n .
    \end{split}
    \end{equation*}
\end{proof}

\subsection{Intermediate part}

Let us begin by presenting two $L^p$-estimates for
the operators $tLe^{-tL}$.
\begin{lemma}
\label{Lpest}
  The family $(tLe^{-tL})_{t>0}$ is uniformly 
  bounded on $L^p(\gamma)$ for all $p > 1$, that is,
  \begin{equation*}
    \sup_{t>0} \| tL e^{-tL} \|_{p\to p} < \infty . 
  \end{equation*}
  Moreover, for $1\leq p \leq 2$ we have
  \begin{equation*}
    \| 1_{E'} tL e^{-tL} 1_E \|_{p\to p}
    \lesssim t^{-n/2} \exp \Big( -\frac{d(E,E')^2}{8t} \Big)
    \sup_{\substack{x\in E \\ y\in E'}} \exp \Big( \frac{|x|^2 + |y|^2}{2} \Big) , \quad 0 < t \leq 1,
  \end{equation*}
  whenever $E,E'\subset\mathbb{R}^n$.
\end{lemma}
\begin{proof}
  The boundedness of $tLe^{-tL}$ on $L^p(\gamma)$ (when $p>1$)
  is the content of \cite[Theorem 5.41]{Jan97}, and the uniformity
  in $t>0$ follows by careful inspection of the proof.
  
  The off-diagonal estimate for $1_{E'}tLe^{-tL}1_E$ is
  an immediate consequence of \cite[Lemma 3]{Kem16} and follows by
  H\"older's inequality.
\end{proof}

Let us then turn to
\begin{equation*}
  \pi_2 f = \int_0^{\widetilde{m}(\cdot)/\kappa} \widetilde{\Phi}(t^2) t^2L
               e^{-\delta' t^2L} (1_{D^c}(\cdot , t) t^2L e^{-\delta t^2L}f) \, 
               \frac{dt}{t} .
\end{equation*}
%with $\kappa \geq 1$ and $0 < \delta, \delta' \leq 1$.

\begin{prop}
  Let $1 < p \leq 2$.
  For $\kappa \geq 4$ and sufficiently small $\delta, \delta' > 0$ we have
  $\| \pi_2 f \|_p \lesssim \| f \|_p$.
\end{prop}
\begin{proof}
  As in \cite[Proposition 5]{Kem16} we have
  \begin{equation}
  \label{pi2dec}
    \| \pi_2 f \|_p 
    \lesssim \sum_{k=2}^\infty \sum_{l=1}^\infty
    \int_{2^{-k-1}}^{2^{-k}} \| 1_{B(0,2^{k-2})} t^2L e^{-\delta' t^2L}
    (1_{C_{k+l-1}} t^2Le^{-\delta t^2L}f) \|_p \, \frac{dt}{t} ,
  \end{equation}
  where $C_{k+l-1} := B(0,2^{k+l-1})\setminus B(0,2^{k+l-2})$.
  
  The distance between $B(0,2^{k-2})$ and $C_{k+l-1}$ is
  at least $2^{k+l-3}$. We make use of Lemma \ref{Lpest} 
  to see that, for $2^{-k-1} < t \leq 2^{-k}$ we have
  \begin{equation*}
    \begin{split}
      &\| 1_{B(0,2^{k-2})} t^2L e^{-\delta' t^2L} (1_{C_{k+l-1}}t^2Le^{-\delta t^2L} f) \|_p \\
      &\lesssim t^{-n} \exp \Big( -\frac{4^{k+l-3}}{8\delta' t^2} \Big)
      \exp \Big( \frac{4^{k-2} + 4^{k+l-1}}{2} \Big) \| t^2Le^{-\delta t^2L} f \|_p \\
      &\lesssim 2^{kn} \exp \Big( -\frac{4^{2k+l-5}}{\delta'} + 4^{k+l-1} \Big) \| f \|_p \\
      &\lesssim \exp (-4^{k+l}) \| f \|_p ,
    \end{split}
  \end{equation*}
  when $\delta' < 4^{-3}$. 
  
  The right-hand side of \eqref{pi2dec} is therefore dominated by
  \begin{equation*}
    \sum_{k=2}^\infty \sum_{l=1}^\infty \exp (-4^{k+l}) \| f \|_p \int_{2^{-k-1}}^{2^{-k}} \frac{dt}{t}
    \lesssim \| f \|_p .
  \end{equation*} 
\end{proof}

Notice that for $p > 1$ the proof was simpler than for $p = 1$
because of the uniform $L^p$-boundedness of $tL e^{-tL}$.

\subsection{Non-admissible part}

We begin by recalling the following key result. 
See \cite[Chapter V]{Jan97} and 
\cite{Gro06} for more references.

\begin{hc}[E. Nelson \cite{Nel73}]
  Let $1 < p \leq q < \infty$. Then
  \begin{equation*}
    \| e^{-tL} \|_{p\to q} \leq 1 , \quad \textup{whenever } t \geq \frac{1}{2}\log \frac{q-1}{p-1} .
  \end{equation*}
\end{hc}
Let us remark, that most proofs of this result use a different
scaling/normalization of the Gaussian measure. The easiest way to
convince oneself of the validity of this version is probably by
the equivalence between hypercontractivity and a logarithmic
Sobolev inequality (see \cite{Gro06}). Also note that
our $L$ is `one half' of a usual Dirichlet form operator.

The following reformulation of the hypercontractivity theorem
will be convenient for us:

\emph{Let $p > 1$. Then for any $t>0$,}
\begin{equation}
  \label{hcreform}
  \| e^{-tL} \|_{p\to q(t)} \leq 1 , \quad 
  \emph{\text{with the hypercontractive exponent }}
  q(t) = 1 + (p-1)e^{2t} .
\end{equation}

Finally, let us consider
\begin{equation*}
  \pi_3f = \int_{\widetilde{m}(\cdot)/\kappa}^\infty 
  \widetilde{\Phi}(t^2) (t^2L)^2 
               e^{-(\delta'+\delta)t^2L}f \, \frac{dt}{t} .
\end{equation*}

\begin{prop}
\label{pi3}
  For sufficiently small $\delta, \delta' > 0$ and large enough $\kappa$ we have:
  \begin{itemize}
    \item If $\Phi$ satisfies Condition D, then $\| \pi_3 f \|_1 \lesssim \| (1 + \log_+ |\cdot |) \, f \|_1$.
    \item If $\Phi$ satisfies Condition P, then $\| \pi_3 f \|_p \lesssim \| f \|_p$ for $1 < p \leq 2$.
  \end{itemize}
\end{prop}
\begin{proof}
  We will consider the two statements side by side.

  \textbf{Part I}:
  Recall the pointwise estimate from \cite[Proposition 7]{Kem16}:
  \begin{equation}
  \label{pi3dec}
    \begin{split}
    | \pi_3 f | &\lesssim \sup_{t>0} |\Phi (t)| \, \Big| (tLe^{-(\delta'+\delta)tL}f)|_{t=\widetilde{m}(\cdot)^2/\kappa^2} \Big| \\
    &\quad + \sup_{t>0} ( |\Phi(t)| + t |\Phi'(t)|) \, \Big| (e^{-(\delta'+\delta)tL}f)|_{t=\widetilde{m}(\cdot)^2/\kappa^2} \Big| \\
    &\quad + \int_{\widetilde{m}(\cdot)^2/\kappa^2}^\infty ( |\Phi'(t)| + t |\Phi''(t)|) \, | e^{-(\delta' + \delta)tL}f | \, dt .
    \end{split}
  \end{equation}
  We will estimate the $L^p$-norms of the three terms on the
  right-hand side separately.
  
  For $1\leq p \leq 2$ and $\kappa$ large enough we have
  \begin{equation*}
    \Big\| (tLe^{-(\delta'+\delta)tL}f)|_{t=\widetilde{m}(\cdot)^2/\kappa^2} \Big\|_p
    + \Big\| (e^{-(\delta'+\delta)tL}f)|_{t=\widetilde{m}(\cdot)^2/\kappa^2} \Big\|_p \lesssim \| f \|_p ,
  \end{equation*}
  by an immediate generalization of \cite[Lemma 6]{Kem16}.
  Together with the general assumption \eqref{phibounds} on
  $\Phi$, this takes care of the first two terms
  of \eqref{pi3dec}.
  
  The range $\int_1^\infty dt$ in the third term in \eqref{pi3dec} 
  is dealt with Conditions D and P separately.
  For $p=1$, Condition D guarantees that
  \begin{equation*}
     \int_1^\infty ( |\Phi'(t)| + t |\Phi''(t)|) \, \| e^{-(\delta'+\delta)tL} f \|_1 \, dt \lesssim \| f \|_1 .
  \end{equation*}
  For $1 < p \leq 2$ we may use interpolation to see that
  $\| e^{-tL} f \|_p \lesssim e^{-\theta_p t} \| f \|_p$ (recall that $f$ is a polynomial with zero mean). 
  Indeed, denoting by $E_0$ the spectral projection onto the kernel of $L$, we have
  $\| e^{-tL} (I-E_0) \|_{2\to 2} = e^{-t}$ and $\| e^{-tL}(I-E_0) \|_{1\to 1} \leq 1$. Hence we obtain the claim with
  $\theta_p = 2 - 2/p$. Now Condition P implies that
  \begin{equation*}
     \int_1^\infty ( |\Phi'(t)| + t |\Phi''(t)|) \, \| e^{-(\delta'+\delta)tL} f \|_p \, dt 
     \lesssim \Big( \int_1^\infty t^N e^{-\theta_p (\delta' + \delta)t} \, dt \Big) \| f \|_p \lesssim \| f \|_p .
  \end{equation*}

  It remains to consider the range 
  $\int_{\widetilde{m}(\cdot)^2/\kappa^2}^1 dt$ in the third term in
  \eqref{pi3dec}. By the assumption \eqref{phibounds},
  $\sup_{0 < t \leq 1} (t|\Phi'(t)| + t^2|\Phi''(t)|) < \infty$,
  and so for $1\leq p \leq 2$ we have
  \begin{equation*}
    \Big\| \int_{\widetilde{m}(\cdot)^2/\kappa^2}^1 ( |\Phi'(t)| + t |\Phi''(t)|) \, | e^{-(\delta' + \delta)tL}f | \, dt \Big\|_p
    \lesssim \Big\| \int_{\widetilde{m}(\cdot)^2/\kappa^2}^1 
    | e^{-(\delta' + \delta)tL}f | \, \frac{dt}{t} \Big\|_p ,
  \end{equation*}
  which is analyzed further below.

  \textbf{Part II (setup)}:
  We will then examine the remaining integral
  over annuli and separate the off-diagonal and on-diagonal parts.
  More precisely, let us write
  $C_0 = B(0,1)$ and $C_k = B(0,2^k)\setminus B(0,2^{k-1})$ for 
  $k\geq 1$, and
  let $C_0^* = B(0,2)$, $C_1^* = B(0,4)$, and
  $C_k^* = B(0,2^{k+1})\setminus B(0,2^{k-2})$ for $k\geq 2$.
  Then, for $1\leq p \leq 2$, we have
  \begin{equation}
  \label{partIIdec}
    \begin{split}
    \Big\| \int_{\widetilde{m}(\cdot)^2/\kappa^2}^1 | e^{-(\delta' + \delta)tL}f| \, \frac{dt}{t} \Big\|_p^p
    &= \sum_{k=0}^\infty \Big\| 1_{C_k} \int_{4^{-k}/\kappa^2}^1 |e^{-(\delta' + \delta) tL} f| \, \frac{dt}{t} \Big\|_p^p \\
    &\lesssim \sum_{k=0}^\infty \Big( \int_{4^{-k}/\kappa^2}^1 \| 1_{C_k} e^{-(\delta' + \delta)tL} 
    (1_{\R^n \setminus C_k^*}f) \|_p \, \frac{dt}{t} \Big)^p \\
    &\quad\quad + \sum_{k=0}^\infty \Big\| 1_{C_k} \int_{4^{-k}/\kappa^2}^1 |e^{-(\delta' + \delta)tL} (1_{C_k^*}f) | \, \frac{dt}{t} \Big\|_p^p .
    \end{split}
  \end{equation}

  \textbf{Part II (off-diagonal terms)}:
  Let $1\leq p \leq 2$ for the time being.
  We choose $\delta, \delta' > 0$ such that 
  $8(\delta' + \delta) \leq 4^{-3}$
  and take care of the first two annuli with $k = 0,1$ simply by
  \begin{equation*}
    \int_{4^{-k}/\kappa^2}^1 \| 1_{C_k} e^{-(\delta' + \delta)tL} (1_{\R^n\setminus C_k^*}f) \|_p \, \frac{dt}{t}
    \leq \Big( \int_{4^{-k}/\kappa^2}^1 \frac{dt}{t} \Big) \| f \|_p \lesssim (k + 1) \| f \|_p .
  \end{equation*}
  
  For the general case with $k \geq 2$ we write
  \begin{equation*}
    \mathbb{R}^n \setminus C_k^* = B(0,2^{k-2}) \cup \bigcup_{l=2}^\infty 
    C_{k+l} .
  \end{equation*}
  Observing that $d(C_k , B(0,2^{k-2})) = 2^{k-2}$, we use
  Lemma \ref{Lpest} to obtain for $t \leq 1$ the estimate
  \begin{equation*}
    \begin{split}
    \| 1_{C_k} e^{-(\delta' + \delta)tL} 1_{B(0,2^{k-2})} \|_{p\to p}
    &\lesssim 2^{kn} \exp \Big( - \frac{4^{k-2}}{8(\delta' + \delta) t} \Big)
    \exp \Big( \frac{4^k + 4^{k-2}}{2} \Big) \\
    &\leq 2^{kn} \exp ( - 4^{k+1} + 4^k ) \\
    &\lesssim \exp (-4^k) .
    \end{split}
  \end{equation*} 
  Furthermore, since $d(C_k , C_{k+l}) = 2^{k+l-2}$,
  Lemma \ref{Lpest} implies that for $t \leq 1$ we have
  \begin{equation*}
  \begin{split}
     \| 1_{C_k} e^{-(\delta' + \delta)tL} 1_{C_{k+l}} \|_{p\to p}
     &\lesssim 2^{kn} \exp \Big( -\frac{4^{k+l-2}}{8(\delta' + \delta)t}\Big) \exp \Big( \frac{4^k + 4^{k+l}}{2} \Big) \\
     &\leq 2^{kn} \exp ( -4^{k+l+1} + 4^{k+l}) \\
     &\lesssim \exp (-4^{k+l}) .
     \end{split}
  \end{equation*}
  
  We are now ready to estimate the off-diagonal terms for $k\geq 2$:
  \begin{equation*}
    \begin{split}
    &\int_{4^{-k}/\kappa^2}^1 \| 1_{C_k} e^{-(\delta' + \delta)tL} (1_{\R^n \setminus C_k^*}f) \|_p \, \frac{dt}{t} \\
    &\leq \int_{4^{-k}/\kappa^2}^1 \| 1_{C_k} e^{-(\delta' + \delta)tL} (1_{B(0,2^{k-2})}f) \|_p \, \frac{dt}{t} \\
    &\quad + \sum_{l=2}^\infty 
    \int_{4^{-k}/\kappa^2}^1 \| 1_{C_k} e^{-(\delta' + \delta)tL} (1_{C_{k+l}}f) \|_p \, \frac{dt}{t} \\
    &\lesssim \Big( \int_{4^{-k}/\kappa^2}^1 \frac{dt}{t} \Big) \exp(-4^k) \| f \|_p
    + \Big( \int_{4^{-k}/\kappa^2}^1 \frac{dt}{t} \Big) \sum_{l=2}^\infty \exp(-4^{k+l}) \| f \|_p \\
    &\lesssim (k+1) \exp(-4^k) \| f \|_p
    \end{split}
  \end{equation*}
  so that the sum of the off-diagonal terms in \eqref{partIIdec} 
  is under control
  \begin{equation*}
    \sum_{k=0}^\infty \Big( \int_{4^{-k}/\kappa^2}^1 \| 1_{C_k} e^{-(\delta' + \delta)tL} (1_{\R^n \setminus C_k^*}f) \|_p \, \frac{dt}{t} \Big)^p \lesssim \| f \|_p^p .
  \end{equation*}
  
  \textbf{Part II (on-diagonal terms)}:
  We then consider the on-diagonal terms in \eqref{partIIdec}.
  
  Let us begin with $p=1$ and estimate for $k\geq 0$ 
  simply as follows:
  \begin{equation*}
  \begin{split}
    \Big\| 1_{C_k} \int_{4^{-k}/\kappa^2}^1 | e^{-(\delta' + \delta)tL} (1_{C_k^*}f) | \, \frac{dt}{t} \Big\|_1
    &\leq \int_{4^{-k}/\kappa^2}^1 \| 1_{C_k} e^{-(\delta' + \delta)tL} (1_{C_k^*}f) \|_1 \, \frac{dt}{t} \\
    &\leq \Big( \int_{4^{-k}/\kappa^2}^1 \frac{dt}{t} \Big) \| 1_{C_k^*} f \|_1\\
    &\lesssim (k + 1) \, \| 1_{C_k^*} f \|_1 .
    \end{split}
  \end{equation*}
  For the sum of the on-diagonal terms we then obtain
  \begin{equation*}
    \sum_{k=0}^\infty \Big\| 1_{C_k} \int_{4^{-k}/\kappa^2}^1 |e^{-(\delta' + \delta)tL}(1_{C_k^*}f)| \, \frac{dt}{t} \Big\|_1
    \lesssim \sum_{k=0}^\infty (k + 1) \, \| 1_{C_k^*} f \|_1
    \eqsim \| (1 + \log_+|\cdot |)\, f \|_1 ,
  \end{equation*}
  as required.
 
  Let then $p > 1$ and 
  choose $\kappa$ to be a power of $4$ and write 
  $N(k) = k-1+ 2\log_4 \kappa$ so that $4^{-k+N(k)+1}/\kappa^2 = 1$ for each $k\geq 0$. We start by partitioning the time integral as
  follows:
  \begin{equation}
  \label{ondiagdec}
    \begin{split}
    &\sum_{k=0}^\infty \Big\| 1_{C_k} \int_{4^{-k}/\kappa^2}^1 |e^{-(\delta' + \delta)tL}(1_{C_k^*}f)| \, \frac{dt}{t} \Big\|_p^p \\
    &= \sum_{k=0}^\infty \sum_{j=0}^{N(k)} \Big\| 1_{C_k} \int_{4^{-k+j}/\kappa^2}^{4^{-k+j+1}/\kappa^2}
    |e^{-(\delta' + \delta)tL}(1_{C_k^*}f)| \, \frac{dt}{t} \Big\|_p^p \\
    &\leq \sum_{k=0}^\infty \sum_{j=0}^{N(k)} \Big( \int_{4^{-k+j}/\kappa^2}^{4^{-k+j+1}/\kappa^2}
    \| 1_{C_k} e^{-(\delta' + \delta)tL}(1_{C_k^*}f) \|_p \, \frac{dt}{t} \Big)^p .
    \end{split}
  \end{equation}
  For each $k,j\geq 0$ let us denote by $q(k,j)$ the hypercontractive exponent (cf. \eqref{hcreform}) from time 
  $(\delta' + \delta) 4^{-k+j} / \kappa^2$, i.e.
  $q(k,j) = 1 + (p-1)e^{2(\delta' + \delta) 4^{-k+j}/\kappa^2}$. Then, using H\"older's inequality, we have for
  $t \geq 4^{-k+j}/\kappa^2$,
  \begin{equation*}
    \| 1_{C_k} e^{-(\delta' + \delta) tL} (1_{C_k^*}f) \|_p
    \leq \gamma(C_k)^{\frac{1}{p} - \frac{1}{q(k,j)}} \| 1_{C_k} e^{-(\delta' + \delta)tL} (1_{C_k^*}f) \|_{q(k,j)}
    \lesssim e^{-c4^j} \| 1_{C_k^*} f \|_p ,
  \end{equation*}
  where the decay factor from the last inequality will be justified next. Firstly,
  \begin{equation*}
    \gamma (C_k) \lesssim \int_{2^{k-1}}^\infty e^{r^2} r^{n-1} \, dr \lesssim e^{-c4^k} .
  \end{equation*}
  Secondly,
  \begin{equation*}
    \frac{1}{p} - \frac{1}{q(k,j)} = \frac{p-1}{p} 
    \frac{e^{2(\delta' + \delta)4^{-k+j}/\kappa^2} - 1}{1 + (p-1)e^{2(\delta' + \delta)4^{-k+j}/\kappa^2}}
    \gtrsim e^{2(\delta' + \delta) 4^{-k+j}/\kappa^2} - 1 \gtrsim 4^{-k+j} .
  \end{equation*}
  Hence,
  \begin{equation*}
    \gamma (C_k)^{\frac{1}{p} - \frac{1}{q(k,j)}} \lesssim (e^{-c4^k})^{4^{-k+j}} \lesssim e^{-c4^j},
  \end{equation*}
  as was claimed.
  
  Returning to the sum of the on-diagonal terms in
  \eqref{ondiagdec},
  \begin{equation*}
  \begin{split}
    &\sum_{k=0}^\infty \sum_{j=0}^{N(k)} \Big( \int_{4^{-k+j}/\kappa^2}^{4^{-k+j+1}/\kappa^2}
    \| 1_{C_k} e^{-(\delta' + \delta)tL}(1_{C_k^*}f) \|_p \, \frac{dt}{t} \Big)^p \\
    &\lesssim \sum_{k=0}^\infty \| 1_{C_k^*} f \|_p^p 
    \sum_{j=0}^{N(k)} \Big( \int_{4^{-k+j}/\kappa^2}^{4^{-k+j+1}/\kappa^2} \frac{dt}{t} \Big)^p e^{-cp4^j} \\
    &\lesssim \sum_{k=0}^\infty \| 1_{C_k^*} f \|_p^p \lesssim \| f \|_p^p .
    \end{split}
  \end{equation*}
  This finishes the proof.
\end{proof}

\begin{remark}
  As is clear from the proof above, if one could show that
  there exists an $\alpha > 1$ such that for all $k\geq 0$ and
  all $0\leq j \leq N(k)$,
  \begin{equation*}
    \| 1_{C_k} e^{-tL} (1_{C_k^*}f) \|_1
    \lesssim j^{-\alpha} \, \| 1_{C_k^*} f \|_1 , \quad t\gtrsim
    4^{-k+j},
  \end{equation*}
  then the desired inequality $\| \pi_3 f \|_1 \lesssim \| f \|_1$
  would follow (for multipliers satisfying Condition D).
\end{remark}

\begin{acks}
  The author would like to thank the organizers of the Probabilistic
  Aspects of Harmonic Analysis conference in Bedlewo, Poland, in
  May 2016. The key steps of the proof were completed in
  the inspiring environment of the Banach Center.
\end{acks}

%\bibliographystyle{plain}
%\bibliography{viitteet}

\begin{thebibliography}{99}

\bibitem[1]{AK15}
A.~Amenta and M.~Kemppainen.
\newblock Non-uniformly local tent spaces.
\newblock {\em Publ. Mat.}, 59(1):245--270, 2015.

\bibitem[2]{AHM12}
P.~Auscher, S.~Hofmann, and J.~M. Martell.
\newblock Vertical versus conical square functions.
\newblock {\em Trans. Amer. Math. Soc.}, 364(10):5469--5489, 2012.

\bibitem[3]{CD13}
A.~Carbonaro and O.~Dragi{\v{c}}evi{\'{c}}.
\newblock Functional calculus for generators of symmetric contraction
  semigroups.
\newblock 2013.
\newblock arXiv:1308.1338.

\bibitem[4]{Cow83}
M.~G. Cowling.
\newblock Harmonic analysis on semigroups.
\newblock {\em Ann. of Math. (2)}, 117(2):267--283, 1983.

\bibitem[5]{Gro06}
L.~Gross.
\newblock Hypercontractivity, logarithmic {S}obolev inequalities, and
  applications: a survey of surveys.
\newblock In {\em Diffusion, quantum theory, and radically elementary
  mathematics}, volume~47 of {\em Math. Notes}, pages 45--73. Princeton Univ.
  Press, Princeton, NJ, 2006.

\bibitem[6]{Jan97}
S.~Janson.
\newblock {\em Gaussian {H}ilbert spaces}, volume 129 of {\em Cambridge Tracts
  in Mathematics}.
\newblock Cambridge University Press, Cambridge, 1997.

\bibitem[7]{JMP14}
M.~Junge, T.~Mei, and J.~Parcet.
\newblock An invitation to harmonic analysis associated with semigroups of
  operators.
\newblock In {\em Harmonic analysis and partial differential equations}, volume
  612 of {\em Contemp. Math.}, pages 107--122. Amer. Math. Soc., Providence,
  RI, 2014.

\bibitem[8]{Kem16}
M.~Kemppainen.
\newblock An ${L}^1$-estimate for certain spectral multipliers associated with
  the {O}rnstein--{U}hlenbeck operator.
\newblock {\em J. Fourier Anal. Appl.}, (to appear), 2016.

\bibitem[9]{MvNP11}
J.~Maas, J.~van Neerven, and P.~Portal.
\newblock Conical square functions and non-tangential maximal functions with
  respect to the {G}aussian measure.
\newblock {\em Publ. Mat.}, 55(2):313--341, 2011.

\bibitem[10]{MvNP12}
J.~Maas, J.~van Neerven, and P.~Portal.
\newblock Whitney coverings and the tent spaces {$T^{1,q}(\gamma)$} for the
  {G}aussian measure.
\newblock {\em Ark. Mat.}, 50(2):379--395, 2012.

\bibitem[11]{MM07}
G.~Mauceri and S.~Meda.
\newblock {${\rm BMO}$} and {$H^1$} for the {O}rnstein-{U}hlenbeck operator.
\newblock {\em J. Funct. Anal.}, 252(1):278--313, 2007.

\bibitem[12]{Mei08}
T.~Mei.
\newblock Tent spaces associated with semigroups of operators.
\newblock {\em J. Funct. Anal.}, 255(12):3356--3406, 2008.

\bibitem[13]{Nel73}
E.~Nelson.
\newblock The free {M}arkoff field.
\newblock {\em J. Funct. Anal.}, 12(2):211--227, 1973.

\bibitem[14]{Por14}
P.~Portal.
\newblock Maximal and quadratic {G}aussian {H}ardy spaces.
\newblock {\em Rev. Mat. Iberoam.}, 30(1):79--108, 2014.

\bibitem[15]{Sjo97}
P.~Sj{\"o}gren.
\newblock Operators associated with the {H}ermite semigroup---a survey.
\newblock In {\em Proceedings of the conference dedicated to {P}rofessor
  {M}iguel de {G}uzm\'an ({E}l {E}scorial, 1996)}, volume~3, pages 813--823,
  1997.

\bibitem[16]{Ste70}
E.~M. Stein.
\newblock {\em Topics in harmonic analysis related to the {L}ittlewood-{P}aley
  theory.}
\newblock Annals of Mathematics Studies, No. 63. Princeton University Press,
  Princeton, N.J., 1970.

\end{thebibliography}
\def\cprime{$'$} \def\cprime{$'$}

\end{document}